\documentclass[12pt]{amsart}
\usepackage{amsmath}
\usepackage{amsfonts}
\usepackage{amsthm}
\usepackage{amssymb}
\usepackage{amscd}
\usepackage[all]{xy}
\usepackage{enumerate}
\usepackage{hyperref}
\usepackage{mathrsfs}

\keywords{} 

\subjclass[2010]{}

\newcommand*{\ext}{\mathcal{E}\kern -.7pt xt}

\newcommand*{\hhom}{\mathcal{H}\kern -.7pt om}
\theoremstyle{plain}
\newtheorem{thm}{Theorem}[section]
\newtheorem{thml}{Theorem}

\newtheorem{prop}[thm]{Proposition}

\newtheorem{cor}[thm]{Corollary}

\newtheorem{lem}[thm]{Lemma}

\theoremstyle{definition}
\newtheorem{defn}[thm]{Definition}

\newtheorem*{ackn}{Acknowledgment}

\newtheorem{rmk}[thm]{Remark}
\newcommand{\sA}{\mathcal{A}}

\newcommand{\sD}{\mathcal{D}}
\newcommand{\sE}{\mathcal{E}}
\newcommand{\sF}{\mathcal{F}}
\newcommand{\sG}{\mathcal{G}}

\newcommand{\sL}{\mathcal{L}}
\newcommand{\sN}{\mathcal{N}}

\newcommand{\sO}{\mathcal{O}}

\newcommand{\sR}{\mathcal{R}}

\newcommand{\sU}{\mathcal{U}}

\newcommand{\sW}{\mathcal{W}}

\newcommand{\mC}{\mathbb{C}}
\newcommand{\mD}{\mathbb{D}}

\newcommand{\mL}{\mathbb{L}}

\newcommand{\mP}{\mathbb{P}}

\newcommand{\mZ}{\mathbb{Z}}

\newcommand{\Ima}{\mathrm{Im}\,}
\newcommand{\Ker}{\mathrm{Ker}\,}

\newcommand{\rank}{\mathrm{rank}\,}

\usepackage{color}

\numberwithin{equation}{section}

\newcommand{\beba}  {\begin{equation}\begin{array}{rcl}}

\newcommand{\eaee}  {\end{array}\end{equation}}

\makeatletter
\def\l@section{\@tocline{1}{0pt}{1pc}{}{}}
\def\l@subsection{\@tocline{2}{0pt}{1pc}{4.6em}{}}
\def\l@subsubsection{\@tocline{3}{0pt}{1pc}{7.6em}{}}
\renewcommand{\tocsection}[3]{%
  \indentlabel{\@ifnotempty{#2}{\makebox[2.3em][l]{%
    \ignorespaces#1 #2.\hfill}}}#3}
\renewcommand{\tocsubsection}[3]{%
  \indentlabel{\@ifnotempty{#2}{\hspace*{2.3em}\makebox[2.3em][l]{%
    \ignorespaces#1 #2.\hfill}}}#3}
\renewcommand{\tocsubsubsection}[3]{%
  \indentlabel{\@ifnotempty{#2}{\hspace*{4.6em}\makebox[3em][l]{%
    \ignorespaces#1 #2.\hfill}}}#3}
\makeatother

\title{Global Kodaira Spencer class and Massey products}

\keywords{Semistable Fibrations, Global Kodaira-Spencer class, Supported deformations, Massey products, Local systems}
\subjclass[2020]{14D06, 14D05, 14E20, 14J40}

\author{Luca Rizzi}
\address{Luca Rizzi\\Department of Mathematics, Computer Science and Physics \\
	Universit\`a di Udine\\
	Udine, 33100\\ Italia
	\texttt{luca.rizzi@uniud.it}}

\author{Francesco Zucconi}
\address{Francesco Zucconi\\Department of Mathematics, Computer Science and Physics \\
Universit\`a di Udine\\
Udine, 33100\\ Italia
\texttt{francesco.zucconi@dimi.uniud.it}}

\begin{document}

\markboth{}{}

\begin{abstract}
We define a new notion of supported global deformation class for a semistable family of complex varieties over a curve $f\colon X\to B$.
We use this notion to study when $X$, possibly up to a finite covering, has a generically finite morphism onto a product $B\times Y$ with $Y$ of general type.
\end{abstract}

\maketitle


\section{Introduction}  Let $f\colon X\to B$ be a semistable family of complex varieties over a complex curve $B$ and with smooth $n-1$-dimensional  general member denoted by $X_b$, $b\in B$. The Kodaira-Spencer map at $b$ identifies a vector subspace inside the space of infinitesimal deformations of $X_b$. 
It is a natural question to study $f\colon X\to B$ in terms of these infinitesimal deformations. In particular, the importance of \textit{supported deformations} in the theory of curves and of fibrations on a surface is well-known; see c.f. \cite[p. 2 and section 6]{AC}, and \cite{PZ}.

In this paper we take a step  forward and we construct a supported \textit{global} deformation class $\rho(\xi)$ naturally given by $f$; see Definition \ref{classe}, \cite{tesi} and also \cite{victor} in the case of a fibered surface. The technical core of the paper is to find the relation between $\rho(\xi)$ and the theory of relative Massey products. Here we recall that the notion of Massey products in algebraic geometry has been introduced in \cite{CP} and \cite{PZ} and then applied in 
\cite{Ra}, \cite{PR}, \cite{CNP}, \cite{victor}, \cite{BGN}, \cite{PT}, \cite{RZ2},
\cite{RZ3}, \cite{RZ1}, \cite{RZ4}, \cite{CRZ} and \cite{R}.  We refer to these sources for a complete discussion and to Section \ref{sez2} for a brief review. 

\subsection{Main results}
Let $\mathbb D^1$ be the local system on $B$ whose elements are the holomorphic  1-forms on the fibers of $f$ which are liftable to \textit{closed} holomorphic forms of $X$ and $L\leq\Gamma(A,\mD^1)$ be a vector space of sections on an open set $A\subseteq B$.  Denote by $s_i\in \Gamma(A,f_*\Omega^1_{X,d})$ liftings to $X$ of the elements of a basis for $L$. 
We define $\sD^A$ as the divisor in $f^{-1}(A)$ given by the common zeroes of the sections $s_{i_1}\wedge\dots\wedge s_{i_{n-1}}\wedge \sigma$ where $\sigma$ runs over the local sections of $\omega_B$.
We denote by $\sD$ the divisor obtained by the horizontal components of $\sD^A$ and with  $D_b$ the restriction of $\sD$ to the general fiber $X_b$. 

The main applications of our point of view are as follows. The first one is strictly related to the class $\rho(\xi)$ and the second one is related to the Albanese morphism.

\begin{thml}
	\label{B}
	Assume that the morphism 
	$
	\bigwedge^{n-1}L\otimes \omega_{B_{|A}}\to f_*\omega_{X_{|A}}
	$ is an injection of sheaves, $\rho(\xi)$ is supported on $\sD$ and that $f_*\sO_X(\sD)$ is a line bundle. Then up to a finite étale covering  $\widetilde{B}\to B$ the associated base change $\widetilde{X}$ has a generically finite surjective morphism  $\widetilde{X}\to\widetilde{B}\times Y$, where $Y$ is an $n-1$-dimensional variety of general type.
\end{thml} See Corollary \ref{finito2}. We will show that this result is based on the possibility to pass from local conditions on $A$ to some global conditions on the finite covering $\widetilde{X}$.

The injectivity of $
	\bigwedge^{n-1}L\otimes \omega_{B_{|A}}\to f_*\omega_{X_{|A}}
	$ is called \textit{strictness}, see Definition \ref{strict}.
We also give an analogue of the divisor $\sD$, a notion of supported deformation and of strictness in the case of $n-1$-forms, instead of 1-forms. The correspondent result is Theorem \ref{C} below, which is technically more difficult to prove. 

In the case of the Albanese morphism we show:

\begin{thml}
	\label{D}
	Let $X$ be a smooth $n$-dimensional variety and $\alpha\colon X\to A:=Alb(X)$ its Albanese morphism. Assume that $\sL:=\Ima(\alpha^*\Omega^{n-1}_A\to \Omega^{n-1}_X)$ is a line bundle on $X$, then the global sections of $\sL$ define a rational map $h\colon X\dashrightarrow Y$ to a variety $Y$ of general type. Furthermore if $H^0(X,\sL)$ is strict, we can take $h$ to be a morphism and $Y$ is the Stein factorization of $X\to Z$ where $Z:=\alpha(X)$. Finally if the restrictions of the Albanese map to the fibers $X_b$ have degree $1$, then these fibers are birational to $Y$. 
\end{thml}

See Theorem \ref{abel} and following Corollary.

\subsection{Global Kodaira-Spencer class}
We recall that 
all the fibers $X_b$ are $n-1$-dimensional and either smooth or reduced and normal crossing divisors. The open set of $B$ corresponding to smooth fibers will be denoted by $B^0$ and its complement $B\setminus B^0$ is the image of the singular fibers.

The exact sequence defining the sheaf of relative differential forms
\begin{equation}
	\label{seqin}
	0\to f^*\omega_B\to\Omega^1_X\to \Omega^1_{X/B}\to 0
\end{equation}
gives the associated extension class $\xi\in\textnormal{Ext}^{1}(\Omega^1_{X/B},f^*\omega_B)$. By restriction of Sequence (\ref{seqin}) to the general fiber we get the sequence
\begin{equation}
	\label{seqin2}
	0\to\sO_{X_b}\otimes T_{B,b}^\vee\to \Omega^1_{X|X_b}\to \Omega^1_{X_b}\to0
\end{equation} 
 and we construct the classes
\begin{equation}
	\xi_b\in \text{Ext}^1(\Omega^1_{X_b},\sO_{X_b})\otimes T_{B,b}^\vee=H^1(X_b,T_{X_b})\otimes T_{B,b}^\vee.
\end{equation} 
where $b\in B^0$.

%

All the extensions $\xi_b$ can be encoded in a unique object thanks to the notion of relative extension sheaf $\ext^1_f(\Omega^1_{X/B},f^*\omega_B)$. Recall that the relative extension sheaf $\ext^p_f$ is by definition the $p$-th derived functor of $f_*\hhom$, hence 
$\ext^1_f(\Omega^1_{X/B},f^*\omega_B)$ is a sheaf on the base $B$ isomorphic to 
$\ext^1_f(\Omega^1_{X/B},\sO_X)\otimes \omega_B$; see cf. \cite{S}. Now by applying the functor $f_*\hhom$ to Sequence (\ref{seqin}) there is  a morphism
$$
\sO_B\to \ext^1_f(\Omega^1_{X/B},\sO_X)\otimes \omega_B.
$$ 
We call \textit{Global Kodaira-Spencer map of the family $f$} the image of $1\in H^0(B,\sO_B)$. It can then be seen as a sheaf morphism
\begin{equation}
		\rho(\xi)\colon T_B\to \ext^1_f(\Omega^1_{X/B},\sO_X)
\end{equation} whose restriction to the general $b\in B$ gives back the usual Kodaira-Spencer map.

The homomorphism
\begin{equation}
	\rho\colon	\textnormal{Ext}^{1}(\Omega^1_{X/B},f^*\omega_B)\to H^0(B,\ext^1_f(\Omega^1_{X/B},f^*\omega_B))
\end{equation} is easily seen to be surjective and it is also an isomorphism when the general fiber $X_b$ is of general type; see Section \ref{sez4} for all the details on this construction. 
\subsection{Globally supported deformation} We refer the reader to Section \ref{sez2} for the notion of Massey triviality; in particular see Definition \ref{mtrivial}. Here we recall that the condition of Massey triviality is a basic tool to study both the vector bundle $K_\partial$ of holomorphic 1-forms on the fibers $X_b$ which are locally liftable to $X$, and the local system $\mathbb D^1$ of holomorphic  1-forms on the fibers which are liftable to \textit{closed} holomorphic forms of $X$. See  \cite{PT}, \cite{RZ4}, \cite{GST}, \cite{GT}.

Now consider $L\leq\Gamma(A,\mD^1)$  and recall that we denote by $\sD$ the divisor obtained by the horizontal components of $\sD^A$ and with  $D_b$ the restriction of $\sD$ to the general fiber $X_b$. 
%

We  define the following sheaf on $A$ $$\ext^{1}_f(\Omega^1_{X/B}(-\sD),f^*\omega_B):=\ext^{1}_f({\Omega^1_{X/B}}_{|f^{-1}(A)}(-\sD),{f^*\omega_{B}}_{|A})$$
and we can finally recall the definition of supported class.
%
%
\begin{defn}	
	\label{classe}
	We say that $\rho(\xi)$ is supported on $\sD$ if 
	\begin{equation}
		\rho(\xi)_{|A}\in \Ker H^0(A,\ext^1_f(\Omega^1_{X/B},f^*\omega_B))\to H^0(A,\ext^{1}_f(\Omega^1_{X/B}(-\sD),f^*\omega_B)).
	\end{equation}
\end{defn}
The following result is a full generalization of \cite[Theorem 1.5.1]{PZ} and \cite[Theorem A]{RZ1}. 
\begin{thml}
	\label{A}
	Let $L\leq\Gamma(A,\mD^1)$ be a Massey trivial vector space. Assume that $L$  generically generates $\Omega^1_{X_b}$ on the general fiber. Then $\rho(\xi)$ is supported on $\sD_{|f^{-1}(A')}$, where $A'\subset A$ is a suitable open dense subset. Viceversa assume that $\rho(\xi)$ is supported on ${\sD}$. If $f_*\sO_X(\sD)$ is a line bundle then the vector space $L$ is Massey trivial.
\end{thml}

See: Theorem \ref{aggiuntarel} and \ref{viceversa}. We point out that actually Theorem \ref{B} is a consequence of this result.


%

In Section \ref{sez5} we construct the theory, parallel to the above one, in the case of volume forms on the fibers $X_b$, that is we consider $n-1$-forms instead of 1-forms. We give an analogue of the divisor $\sD$, that here we denote $\sD^{n-1}$, and a notion of supported deformation in the case of $n-1$-forms. The next Theorem relies on the generalization of Castelnuovo-de Franchis for volume forms proved in \cite[Theorem 7.2]{RZ5} and the suitable notion of strictness used therein and recalled in Section \ref{sez5}.

\begin{thml}
	\label{C}
	Assume that the Global Kodaira-Spencer class of $f$ is $n-1$-supported on $\sD^{n-1}$ and that strictness condition holds. If $f_*\hhom( \Omega^{n-1}_{X/B}(-\sD^{n-1}),\Omega^{n-2}_{X/B})$ is zero then there exists a smooth $n-1$-dimensional variety $Y$ of general type and a generically finite dominant map $f^{-1}(A)\dashrightarrow Y\times B$.
\end{thml}

See Theorem \ref{viceversap}. Regarding the vanishing hypothesis,
note that on the smooth fibers, this vanishing is equivalent to the vanishing of the cohomology group $H^0(X_b, T_{X_b}(D_b^{n-1}))$, where $D_b^{n-1}$ is the restriction of $\sD^{n-1}$ on $X_b$; this can essentially be seen as a hypothesis on the normal sheaf of $D_b^{n-1}$. In particular it can often be applied if $X_b$ is of general type and $D_b^{n-1}$ has negative self-intersection.
A global version of Theorem \ref{C} on a finite base change of $X$ can be found in Corollary \ref{globale}

%
%

We give some bound on the geometric genus of $Y$ in the case of a relatively minimal fibered threefold $f\colon X\to B$ under some hypotheses on the canonical map $\phi_{|K_{X_b}|}$ of the general fiber $X_b$ following  \cite{Ri}.
\subsection{Other results}
Finally we point out that in Section \ref{sez2} we revise the theory of Massey products according a new perspective and this lead us  to show, in Section \ref{sez3}, a relative version of our old theorem on adjoint quadrics \cite[Theorem B]{RZ1}. Indeed we think that Theorem \ref{reladjquadthm} has its own interest as a criterion for Massey triviality and as a tool to show finiteness  results on certain monodromy groups, see Corollary \ref{cornag}.

\begin{ackn}
The first author has been supported by the IBS Center for Complex Geometry and by European Union funds, NextGenerationEU. The second author has been supported by the grant DIMA Geometry PRIDZUCC and by PRIN 2017 Prot. 2017JTLHJR \lq\lq Geometric, algebraic and analytic methods in arithmetics\rq\rq.
\end{ackn}

\section{Massey products and local systems}
\label{sez2}
In this section we briefly recall and discuss the main constructions of \cite{RZ4}, in particular we give the rigorous definition of the vector bundle $K_\partial$, the local system $\mD^1$ and the notion of Massey triviality mentioned in the Introduction.
\subsection{Local systems of certain liftable holomorphic forms}

Let $X$ be a smooth complex compact $n$-dimensional variety and $B$ a smooth complex curve.
From the Introduction we recall that we consider semistable fibrations $f\colon X\to B$ where  $X_b=f^{-1}(b)$ denotes the fiber over a point $b\in B$. All the fibers $X_b$ are either smooth or reduced and normal crossing divisors.
Let $B_0$ be the locus of singular values of $f$ and $B^0 = B \setminus B_0$ the open set of regular values. Consider the exact sequence 
\begin{equation}
	\label{omegarel}
	0\to f^*\omega_B\to \Omega^1_{X}\to \Omega^1_{X/B}\to0
\end{equation} defining the sheaf of relative differentials $\Omega^1_{X/B}$. It is not difficult to see that, under our hypothesis on $f$, the sheaf $\Omega^1_{X/B}$ is torsion free but not locally free in general. 

The $p$-wedge of Sequence (\ref{omegarel}) is 
\begin{equation}
	\label{omegarelp}
	0\to f^*\omega_B\otimes \Omega^{p-1}_{X/B}\to \Omega^p_{X}\to \Omega^p_{X/B}\to0
\end{equation} which is also exact for $f$ semistable. Note that for $p=n-1$, the sheaf $\Omega^{n-1}_{X/B}$ injects into the relative dualizing sheaf $\omega_{X/B}:=\omega_X\otimes f^*T_B=\omega_X\otimes f^*\omega_B^\vee$; their direct images via $f$ are isomorphic when restricted to $B^0$.

Taking the pushforward of Sequence (\ref{omegarel}) we obtain the long exact sequence on $B$
\begin{equation}
0\to \omega_B\to f_*\Omega^1_X\to f_*\Omega^1_{X/B}\to R^1f_*\sO_X\otimes \omega_B\to \dots
\end{equation} and we call $K_\partial$ the cokernel in the exact sequence
\begin{equation}
	\label{kdelta}
	0\to \omega_B\to f_*\Omega^1_X\to K_\partial\to 0.
\end{equation} 
 Intuitively we can think of $K_\partial$ as the vector bundle of holomorphic 1-forms on the fibers of $f$ which are locally liftable to the variety $X$.
A key property of $K_\partial$ is given in the following Lemma, see  \cite[Lemma 3.5]{PT} or \cite[Lemma 2.2]{RZ4}.
\begin{lem}
	\label{splitlemma}
	If $f\colon X\to B$ is a semistable fibration, the exact sequence
	\begin{equation}
		0\to \omega_B\to f_*\Omega^1_X\to K_\partial\to 0
	\end{equation} splits.
\end{lem}
This means that the above intuitive idea is not only true locally around the fibers, but the liftability holds on every open subset of $B$. For a more complete study of $K_\partial$ see \cite{PT}, \cite{GST}, \cite{GT} for the case $n=2$, \cite{RZ4} for the general case.

If in Sequence (\ref{kdelta}) we consider, instead of $\Omega^1_X$, the sheaf $\Omega^1_{X,d}$ of de Rham closed differential forms, we obtain the exact sequence 
	\begin{equation}
		\label{seqd1}
	0\to \omega_B\to f_*\Omega^1_{X,d}\to \mD^1\to 0.
\end{equation} It turns out that $\mD^1$ is a local system on the curve $B$ as shown in \cite{PT} for 1-dimensional fibers and in \cite{RZ4} for any dimension. Note that $\mD^1$ is a subsheaf of $K_\partial$ and we can interpret $\mD^1$ as the local system of holomorphic 1-forms on the fibers of $f$ which are liftable
to \textit{closed} holomorphic forms of the variety $X$. Finally note that, by Lemma \ref{splitlemma},  also the exact Sequence (\ref{seqd1}) splits.

One can define the vector bundles $K_\partial^p$ and the local systems $\mD^p$ in a similar way. 
Take the direct image via $f_*$ of the exact sequence (\ref{omegarelp}) and obtain the long exact sequence 
 \begin{equation}
 	0\to \omega_B\otimes f_*\Omega^{p-1}_{X/B}\to f_*\Omega^p_X\to f_*\Omega^p_{X/B}\to R^1f_*\Omega^{p-1}_{X/B}\otimes \omega_B\to \dots
 \end{equation} We define $K_\partial^p$ and $\mD^p$ respectively by
 \begin{equation}
	\label{seqkp}
	0\to \omega_B\otimes f_*\Omega^{p-1}_{X/B}\to f_*\Omega^p_{X}\to K_\partial^p\to0
\end{equation} and
 \begin{equation}
 	\label{seqdp}
	0\to \omega_B\otimes f_*\Omega^{p-1}_{X/B}\to f_*\Omega^p_{X,d}\to \mD^p\to0
\end{equation} where once again $\Omega^p_{X,d}$ denotes the sheaf of de Rham closed $p$-forms.
As in the case of $p=1$, the sheaves $\mD^p$ are local systems  of holomorphic $p$-forms on the fibers of $f$ which are locally liftable
to \textit{closed} holomorphic $p$-forms on $X$. 
\begin{rmk}
	\label{solle}
	A difference between the case $p=1$ and $p>1$ is that while Sequences (\ref{kdelta}) and (\ref{seqd1})  split, Sequences (\ref{seqkp}) and (\ref{seqdp}) in general do not.
\end{rmk}

By \cite{RZ4}, the local system $\mD^{n-1}$ is of particular interest because it turns out to be the local system of the second Fujita decomposition of $f_*\omega_{X/B}$. In fact, by a famous result of Fujita, see \cite{Fu} and \cite{Fu2}, \cite{CD1}, \cite{CD2}, the direct image $f_*\omega_{X/B}$ is a sum
\begin{equation}
\label{seconda}
	f_*\omega_{X/B}\cong \sU\oplus \sA
\end{equation} where $\sU$ is a unitary flat vector bundle and $\sA$ is ample. It turns out that $\sU=\mD^{n-1}\otimes\sO_B$, see \cite[Theorem 3.7]{RZ4} for the proof. Some analysis on the rank of $\sU$ is in \cite{GT}, \cite{GST} if $X$ is a surface or in \cite{Ri} if $X$ is a threefold, see also Section \ref{sez5} for some applications.

\subsection{Massey products}

Massey products, originally called adjoint forms, have been introduced in \cite{CP} and \cite{PZ}. They have been useful for the study of infinitesimal deformations and also for the study of the monodromy of the above mentioned local systems.

We now recall their construction. This presentation is slightly different but equivalent to the one in \cite{RZ4}, and it will be more convenient for the applications contained in this paper.

According to Lemma \ref{splitlemma}, Sequence (\ref{kdelta}) splits; from now on for simplicity we choose and fix one of these splittings. The following wedge sequence also splits
\begin{equation}
	\label{split2}
	\xymatrix{
		0\ar[r]& \bigwedge^{n-1} K_\partial\otimes \omega_B\ar[r]&\bigwedge^n f_*\Omega^1_X\ar[r]& \bigwedge^n K_{\partial}\ar[r]\ar @/_1.6pc/ [l]_{} & 0
	}
\end{equation}
and we take the composition of this splitting
with the natural wedge map and obtain the morphism
\begin{equation}
	\label{agg}
	\lambda\colon \bigwedge^n K_{\partial}\to \bigwedge^n f_*\Omega^1_X\to f_*\bigwedge^n\Omega^1_X=f_*\omega_X. 
\end{equation}

Now consider  $n$ sections $\eta_1,\ldots,\eta_n\in \Gamma(A,K_{\partial})$ on an open subset $A\subseteq B$; call $s_1,\dots,s_n\in \Gamma(A,f_*\Omega^1_{X})$ liftings of $\eta_1,\dots,\eta_n$ according to the above chosen splitting.

\begin{defn}
	\label{omegai}
	We call $\omega_i$, $i=1,\ldots,n$, the wedge $s_1\wedge\dots\wedge\widehat{s_i}\wedge \dots\wedge s_n\in \Gamma(A,f_*\Omega^{n-1}_X)$ and $\sW$ the submodule of $f_*\omega_X$ generated by $\langle\omega_i\rangle\otimes \omega_B$.
\end{defn}
\begin{defn}
	\label{mtrivial}
	The \emph{Massey product or adjoint image} of $\eta_1,\ldots,\eta_n$ is the section of $f_*\omega_{X}$ computed by $\lambda(\eta_1\wedge...\wedge\eta_n)$. We say that the sections $\eta_1,\ldots,\eta_n$  are \emph{Massey trivial} if their Massey product is contained in the submodule $\sW$.
\end{defn}

\begin{rmk}
	\label{remlif}
	The Massey product is given explicitly by $s_1\wedge\dots\wedge s_n$ and being Massey trivial means that locally
	$$
	s_1\wedge\dots\wedge s_n=\sum_i \omega_i\otimes \sigma_i
	$$ where the $\omega_i$ are as in Definition \ref{omegai} and $\sigma_i$ are local sections of $\omega_B$.
	
	As a section of $f_*\omega_{X}$, the Massey product certainly depends on the choice of the splitting mentioned above. On the other hand, the condition of being Massey trivial does not; see \cite{RZ4}. In Proposition \ref{localglobal} we will show that if the sections $\eta_1,\ldots,\eta_n$  are Massey trivial, there is a very convenient choice for this splitting.
\end{rmk}

In the literature mentioned at the beginning, the construction of Massey products is done pointwise, that is for a fixed regular value $b\in B$ and working on the fiber $X_b$ and on an infinitesimal neighbourhood of this fiber.
It is not difficult to see that all the pointwise defined Massey products can be glued together and this agrees exactly with Definition \ref{mtrivial} on suitable open subsets $A\subset B$.


Of course since $\mD^1$ is a subsheaf of $K_{\partial}$, it makes sense to construct Massey products starting from sections of $\mD^1$, i.e. consider sections $\eta_i\in \Gamma(A,\mD^1)$. One of the key points in \cite{PT} and \cite{RZ4} is exactly to consider this setting.


To conclude this section we recall the notion of strictness and its relation with Massey triviality.
Let $A\subseteq B$ be a contractible open subset and $W\leq \Gamma(A, K_{\partial})$ a vector subspace of dimension at least $n$.

We give the following definition
\begin{defn}
	\label{mastriv}
	We say that $W$ is Massey trivial if any $n$-uple of sections in $W$ is Massey trivial (according to  Definition \ref{mtrivial}).
\end{defn}

\begin{defn}
	\label{strict}
	We say that $W$ is strict if the map 
	$$
	\bigwedge^{n-1}W\otimes \omega_{B_{|A}}\to f_*\omega_{X_{|A}}
	$$ is an injection of sheaves.
\end{defn}

The following proposition shows how Massey triviality and strictness give a preferred choice of liftings as we anticipated in Remark \ref{remlif}.

\begin{prop}
	\label{localglobal}
	Let $W\leq \Gamma(B, K_{\partial})$ be a strict subspace of global sections of $K_{\partial}$ and let $A\subseteq B$ be an open contractible subset. If the sections of $W$ are Massey trivial when restricted to $A$ then there exist a unique lifting $\widetilde{W}\leq \Gamma(B, f_*\Omega^1_X)$ such that 
	$$
	\bigwedge^{n}\widetilde{W}\to \Gamma(B,f_*\omega_{X})
	$$ is zero. If furthermore $W\leq \Gamma(B, \mD^1)$ then $\widetilde{W}\leq \Gamma(B, f_*\Omega^1_{X,d})$.
\end{prop}
For the proof see \cite[Proposition 4.10]{RZ4}.
If the sections $\eta_i\in W$ are Massey trivial, for any choice of liftings $s_i$ we have a relation of the form 
$$
s_1\wedge\dots\wedge s_n=\sum_i \omega_i\otimes \sigma_i
$$ as seen in Remark \ref{remlif}. This proposition tells us that actually there is a preferred choice of liftings $\tilde{s}_i$ such that $$\tilde{s}_1\wedge\dots\wedge\tilde{s}_n=0.$$ It also tells us that local Massey triviality implies global Massey triviality.

\begin{rmk}
	We stress that the strictness condition is essential to prove Proposition \ref{localglobal} if $\dim W>n$.
	We also recall that given an $n$-dimensional variety $Y$, a subspace $W\leq H^0(Y,\Omega^1_Y)$ is usually called strict if the natural map  $\bigwedge^n W\to H^0(Y,\omega_Y)$ is non zero on decomposable elements. See \cite[Definition 2.1 and 2.2]{Ca2}. If the map $\bigwedge^n W\to H^0(Y,\omega_Y)$ is an isomorphism on the image, we have the notion of strongly strict, see \cite[Definition 2.2.1]{RZ1}. 
\end{rmk}

\section{Relative Adjoint quadrics}
\label{sez3}
As a natural continuation of \cite{RZ4}, in this section we study the generalization of the notion of adjoint quadrics, introduced in \cite{RZ1}. As we will see, the presence or absence of certain quadratic relations is strictly related to the notion of Massey triviality.

In the following, consider as before an open subset $A\subseteq B$ and $\eta_1,\dots,\eta_n$ a basis of an $n$-dimensional vector space $W\leq\Gamma(A,\mD^1)$. Choosing a splitting of 
\begin{equation}
	\label{split}
	\xymatrix{
		0\ar[r]&  \omega_B\ar[r]&f_*\Omega^1_{X,d}\ar[r]& \mD^1\ar[r]\ar @/_1.2pc/ [l]_{} & 0
	}
\end{equation}
and $s_1,\dots,s_n\in \Gamma(A,f_*\Omega^1_{X,d})$ liftings of $\eta_1,\dots,\eta_n$ accordingly, we denote by $\omega$ the Massey product of the $\eta_i$. With our choice of liftings, $\omega$ is explicitly given by $s_1\wedge\dots\wedge s_n$. Also recall that by definition $\omega_i:=s_1\wedge\dots\wedge\widehat{s_i}\wedge\dots\wedge s_n.$
We have the following definition
\begin{defn}
	\label{qa}
	A \textit{relative adjoint quadric} is a local quadratic relation of sections of $f_*\omega_{X}\otimes f_*\omega_X$ of the form 
	$$
	\omega^2=\sum (\omega_i\wedge\sigma_i)\cdot \rho_i
	$$ where $\sigma_i$ are local sections of $\omega_B$, $\rho_i$ of  $f_*\omega_{X}$ and $\omega$, $\omega_i$ are as above.
\end{defn}

To study the role of these relations we construct a commutative diagram as follows.
Take the exact sequence 
$$
0\to f^*\omega_B\to \Omega^1_X\to \Omega^1_{X/B}\to 0
$$ and recall that its top wedge 
$$
0\to \Omega^{n-2}_{X/B}\otimes f^*\omega_B\to \Omega^{n-1}_X\to \Omega^{n-1}_{X/B}\to 0
$$
is exact. The tensor product with $f^*\omega_B$ followed by the direct image $f_*$ gives the long exact sequence of sheaves on $B$
\begin{equation}
	0\to f_*\Omega^{n-2}_{X/B}\otimes\omega_B^{\otimes 2} \to  f_*\Omega^{n-1}_X\otimes\omega_B\to  f_*\Omega^{n-1}_{X/B}\otimes\omega_B\to \dots
\end{equation} By the inclusion $f_*\Omega^{n-1}_{X/B}\hookrightarrow f_*\omega_{X/B}$ (which is an isomorphism on $B^0$) we actually obtain an inclusion $ f_*\Omega^{n-1}_{X/B}\otimes\omega_B\hookrightarrow  f_*\omega_{X/B}\otimes\omega_B=f_*\omega_X$ hence we can add the diagonal morphism 
\begin{equation}
	\xymatrix{
0\ar[r]  &  f_*\Omega^{n-2}_{X/B}\otimes\omega_B^{\otimes 2}\ar[r]&  f_*\Omega^{n-1}_X\otimes\omega_B\ar[r]&  f_*\Omega^{n-1}_{X/B}\otimes\omega_B\ar[r]\ar@{^{(}->}[rd]& \dots\\
&&&&f_*\omega_X
}
\end{equation}

We complete the diagram on $A$ with the following second row and appropriate morphisms
\begin{equation}
	\label{diagrammaquadriche}
	\xymatrix{
		0\ar[r] &  f_*\Omega^{n-2}_{X/B}\otimes\omega_B^{\otimes 2}\ar[r]\ar^{\psi}[dd] &  f_*\Omega^{n-1}_X\otimes\omega_B\ar[r]\ar^{\phi_\omega}[dd] &  f_*\Omega^{n-1}_{X/B}\otimes\omega_B\ar[dd] \ar[r]\ar@{^{(}->}[rd]& \dots\\
		&&&&f_*\omega_X\ar^{\cdot\omega}[dl]\\
		0\ar[r]  & K\ar[r]& \bigwedge^{n-1} \mD^1\otimes \omega_B\otimes f_*\omega_X\ar^-{\nu\otimes id}[r]& f_*\omega_X\otimes f_*\omega_X\ar[r]& \dots\\
	}
\end{equation} The maps above are defined as follows.

Firstly $\cdot\omega$ is just the multiplication by the Massey product $\omega$, sending a section $\tau$ of $f_*\omega_X$ to $\tau\cdot\omega$.

The map $\nu\colon \bigwedge^{n-1} \mD^1\otimes \omega_B\to f_*\omega_X$ is given by taking $n-1$ sections of $\mD^1$, call them $\mu_1,\dots,\mu_{n-1}$, liftings of these sections, $t_1,\dots,t_{n-1}$, according to our fixed splitting of Sequence (\ref{split}) and defining $\nu(\mu_1\wedge\dots\wedge\mu_{n-1}\otimes \sigma)=t_1\wedge\dots\wedge t_{n-1}\wedge\sigma$ for $\sigma$ in $\omega_B$. In particular note that
\begin{equation}
\nu(\eta_1\wedge\dots\wedge\widehat{\eta_i}\wedge\dots\wedge\eta_{n}\otimes \sigma)=\omega_i\wedge\sigma.
\label{nu}
\end{equation} 

Finally $\phi_\omega$ is given by the aforementioned liftings $s_1,\dots,s_n$ as follows. Locally, given a section $s$ of $f_*\Omega^{n-1}_X$, the image $\phi_\omega(s\otimes \sigma)$ is
\begin{equation}
\phi_\omega(s\otimes \sigma)=\sum_i (-1)^i \eta_1\wedge\dots\wedge\widehat{\eta_i}\wedge\dots\wedge\eta_{n}\otimes\sigma\otimes  s\wedge s_i.
\label{phi}
\end{equation}

We define $K$ as the kernel of $\nu\otimes id$ and it is easy to see that $\phi_\omega$ restricts to a map $\psi\colon f_*\Omega^{n-2}_{X/B}\otimes\omega_B^{\otimes 2}\to K$ and that the above diagram is commutative.

\begin{defn}
We say that the Massey product $\omega\in\Gamma(A,f_*\omega_X)$ is locally liftable if it is in the image of the sheaf morphism
$$
f_*\Omega^{n-1}_X\otimes \omega_B\to f_*\omega_X
$$ of Diagram \ref{diagrammaquadriche}.
\end{defn}

\begin{rmk}
	Note that in particular $\omega$ is locally liftable if it is an element of $\mD^{n-1}\otimes \omega_B\subset f_*\omega_X$ since the sections of $\mD^{n-1}$ are locally liftable to $f_*\Omega_{X,d}^{n-1}$.  This means that this theory is well suited to approach the natural question of what happens when the Massey product of sections $\eta_i\in \Gamma(A,\mD^1)$ ends up in $\Gamma(A,\mD^{n-1}\otimes \omega_B)$, that is in the part of $f_*\omega_X$ given by the local system of the Fujita decomposition.
\end{rmk}

The generalization of  \cite[Theorem 2.1.2]{RZ1} is:
\begin{thm}
	\label{reladjquadthm}
	Let $f\colon X\to B$ be a semistable fibration. If the Massey product $\omega$ is locally liftable and there are no relative adjoint quadrics, then the sections $\eta_i$ are Massey trivial.
\end{thm}
\begin{proof}

Since the Massey product $\omega$ is locally liftable, we call $\tilde{\omega}_\alpha$ the local lifting of $\omega_{|A_\alpha}$ in $f_*\Omega^{n-1}_X\otimes \omega_B$, $A=\bigcup A_\alpha$ an open covering. The difference between two such liftings is in $f_*\Omega^{n-2}_{X/B}\otimes\omega_B^{\otimes 2}$, hence by the commutativity of the first square of Diagram (\ref{diagrammaquadriche}), we have that $(\nu\otimes id)(\phi_\omega(\tilde{\omega}_\alpha))$ glue together to a section of $\Gamma(A, f_*\omega_X\otimes f_*\omega_X)$ which we will denote,  by abuse of notation, $(\nu\otimes id)(\phi_\omega(\tilde{\omega}))$. 

Consider now the commutative square 
\begin{equation}
	\xymatrix{
		f_*\Omega^{n-1}_X\otimes\omega_B\ar[r]\ar^{\phi_\omega}[d] &  f_*\omega_X \ar^{\cdot\omega}[d]\\
		\bigwedge^{n-1} \mD^1\otimes \omega_B\otimes f_*\omega_X\ar^-{\nu\otimes id}[r]& f_*\omega_X\otimes f_*\omega_X\\
	}
\end{equation} coming from Diagram (\ref{diagrammaquadriche}).
We have that $\omega^2=(\nu\otimes id)(\phi_\omega(\tilde{\omega}))$. Now note that by definition $\phi_\omega(\tilde{\omega}_\alpha)$ is a sum containing the wedges $\eta_1\wedge\dots\wedge\widehat{\eta_i}\wedge\dots\wedge\eta_{n}$ as we have seen in (\ref{phi}). Now applying $\nu$  all these wedges $\eta_1\wedge\dots\wedge\widehat{\eta_i}\wedge\dots\wedge\eta_{n}$ produces the sections $\omega_i$ as seen in (\ref{nu}).

We deduce that $\nu\otimes id(\phi_\omega(\tilde{\omega}))$ is locally of the form $\sum \omega_i\wedge\sigma_i\cdot \rho_i
$ where $\sigma_i$ are local sections of $\omega_B$ and $\rho_i$ of  $f_*\omega_{X}$.
Now assume by contradiction that $\omega$ is not Massey trivial, then the relation $\omega^2=\nu\otimes id(\phi_\omega(\tilde{\omega}))$ is a true quadratic relation (and not just the square of a linear relation) and gives a relative adjoint quadric. By our hypothesis these do not exist hence the contradiction and $\omega$ is Massey trivial.
\end{proof}

The first application of Theorem \ref{reladjquadthm} comes from \cite[Theorem B]{RZ4} and gives information on the monodromy associated to local systems generated by Massey trivial vector spaces. Given a vector subspace $L\leq \Gamma(A,\mD^1)$, $L$  naturally generates a local system in $\mD^1$ by taking the closure under the monodromy action.
More precisely, if we denote by $G$ the monodromy group acting non-trivially on $\mD^1$, the local system generated by $L$ is by definition the local system with stalk $\widehat{L}=\sum_{g\in G}g\cdot L$. We will denote it by $\mL$. 
\begin{defn}
	\label{mastrivgen}
	If $L$ is Massey trivial, we will say that $\mL$ is Massey trivial generated.
	
\end{defn} See \cite[Definition 5.5]{PT}.
Consider the action of the fundamental group $\pi_1(B, b)$ on the stalk of $\mL$ and call  $H_\mL$ the subgroup of $\pi_1(B, b)$ acting trivially on $\mL$ and $G_\mL=\pi_1(B, b)/H_\mL$ the associated monodromy group.
\begin{cor}
	\label{cornag}
Let $L$ be a  strict vector space such that every Massey product of sections of $L$ is locally liftable and there are no relative adjoint quadrics. Then $L$ is Massey trivial and the local system $\mL$ generated under the monodromy action is Massey trivial generated. In particular $\mL$ has finite monodromy.
\end{cor}
\begin{proof}
Take $n$ linearly independent sections of $L$ and consider the associated Massey product. The Massey triviality follows from the previous theorem, hence $L$ is a Massey trivial vector space. The local system $\mL$ generated under the monodromy action is then Massey trivial generated by definition. Local systems generated by a strict and Massey trivial vector space have finite monodromy by \cite[Theorem B]{RZ4}.
\end{proof} For applications of this result see \cite{RZ4}.

\section{Global supported deformations}
\label{sez4}
We recall that originally, see  \cite{RZ2}, \cite{RZ3}, \cite{RZ1},  Massey products have been used as a tool for the study of infinitesimal deformations. Here we generalize this setting  in the case of semistable families $f\colon X\to B$, see also \cite{tesi}, before giving another consequence of Theorem \ref{reladjquadthm}.
\subsection{The Global Kodaira-Spencer map}
Consider again the exact sequence 
\begin{equation}
	\label{relatdiff}
	0\to f^*\omega_B\to \Omega^1_{X}\to\Omega^1_{X/B}\to0.
\end{equation}
The restriction of Sequence (\ref{relatdiff}) on a smooth fiber $X_b$ is the  sequence
\begin{equation}
	\label{seqfibristretta}
	0\to\sO_{X_b}\otimes T_{B,b}^\vee\to \Omega^1_{X|X_b}\to \Omega^1_{X_b}\to0
\end{equation} 
which is associated to an element 
\begin{equation}\label{xib}
	\xi_b\in H^1(X_b,T_{X_b})\otimes T_{B,b}^\vee=\text{Ext}^1(\Omega^1_{X_b},\sO_{X_b})\otimes T_{B,b}^\vee.
\end{equation} Since $ H^1(X_b,T_{X_b})$ is the space of first order deformations of $X_b$, the class $\xi_b$ naturally corresponds to the deformation of the fiber  $X_b$ induced by the family $f\colon X\to B$. The key to encode all the extensions $\xi_b$ in a unique object is the notion of relative extension sheaf. We have learned this tool from \cite{S}.
\begin{defn}
	Given a morphism of schemes $f\colon X\to Y$, the relative extension sheaf $\ext^p_f$ is  the $p$-th derived functor of $f_*\hhom$.
\end{defn} 
For all the properties of the relative extension sheaves we refer to \cite[Chapter 1]{B}. Here we only recall the following:
\begin{thm}
	\label{proprieta}
	The sheaves $\ext^p_f$ satisfy
	\begin{enumerate}
		\item If $f$ is projective and $\sF,\sG$ are coherent $\sO_X$-modules, then $\ext^p_f(\sF,\sG)$ is a coherent $\sO_X$-module
		\item $\ext^p_f(\sF,\sG)$ is the sheaf associated to the presheaf $U\mapsto \textnormal{Ext}^p(\sF|_{f^{-1}(U)},\sG|_{f^{-1}(U)})$. In particular it holds that 
		\begin{equation*}
			\ext^p_f(\sF,\sG)_{|U}\cong\ext^p_f(\sF_{|f^{-1}(U)},\sG_{|f^{-1}(U)})
		\end{equation*} 
		\item $\ext^p_f(\sO_X,\sG)=R^pf_*\sG$
		\item If $\sL$ and $\sN$ are locally free sheaves of finite rank on $X$ and $Y$, respectively, then
		\begin{equation*}
			\begin{split}
				\ext^p_f(\sF\otimes\sL,-\otimes f^*\sN)\cong\ext^p_f(\sF,-\otimes\sL^\vee\otimes f^*\sN)\cong\\\cong\ext^p_f(\sF,-\otimes\sL^\vee)\otimes\sN
			\end{split}
		\end{equation*}
		\item For any $\sO_X$-modules $\sF,\sG$ there is a spectral sequence, called  \emph{local to global spectral sequence},
		\begin{equation*}
			E_2^{p,q}=R^pf_*\ext^q(\sF,\sG)\Longrightarrow\ext^{p+q}_f(\sF,\sG)
		\end{equation*}where $\ext^q$ is the usual extension sheaf on $X$, that is the derived functor of $\hhom$.
	\item Under the same hypotheses of (5), we also have the spectral sequence 
	\begin{equation*}
		E_2^{p,q}=H^p(B,\ext^{q}_f(\sF,\sG))\Longrightarrow\textnormal{Ext}^{p+q}(\sF,\sG).
	\end{equation*}
	\end{enumerate}
\end{thm} 
The spectral sequences in (5) and (6) can both be seen as a consequence of a result of  Grothendieck that computes the derived functor of the composition of two functors $F$ and $G$ knowing the derived functors of $F$ and $G$ separately, cf. \cite[Theorem 12.10]{Mc}. In (5) we take $F=f_*$ and $G=\hhom$ and in (6) $F=\Gamma$ and $G=f_*\hhom$.

Now if we apply the functor $f_*\hhom(-,f^*\omega_B)$ to the exact Sequence (\ref{relatdiff}) we obtain, from the resulting long exact sequence, the morphism
$$
f_*\hhom(f^*\omega_B,f^*\omega_B)\to \ext^1_f(\Omega^1_{X/B},f^*\omega_B)
$$ which translates, by the properties mentioned in Theorem \ref{proprieta}, into
$$
\sO_B\to \ext^1_f(\Omega^1_{X/B},\sO_X)\otimes\omega_B.
$$ 
\begin{defn}
	\label{kodspen}
	The image of $1\in H^0(B,\sO_B)$ is a morphism
	\begin{equation}
		\label{ksf}
		T_B\to  \ext^1_f(\Omega^1_{X/B},\sO_X)
	\end{equation} which is called the Global Kodaira-Spencer map.
\end{defn}

In this paper, we will mainly consider the extension sheaf
\begin{equation*}
	\ext^1_f(\Omega^1_{X/B},f^*\omega_B)=\ext^1_f(\Omega^1_{X/B},\sO_X)\otimes\omega_B.
\end{equation*}
The following Lemma shows how this sheaf behaves on a suitable Zariski open set $B'\subset B$ and justifies the name Kodaira-Spencer for the morphism in Definition (\ref{kodspen}).

\begin{lem}
	\label{generalg}
	There is an injection
	\begin{equation*}
		R^1f_*\hhom(\Omega^1_{X/B},f^*\omega_B)\hookrightarrow \ext^1_f(\Omega^1_{X/B},f^*\omega_B)
	\end{equation*}
	which is an isomorphism over an open dense subset of $B$. In particular, for general $b\in B$ we have the isomorphism
	\begin{equation*}
		\ext^1_f(\Omega^1_{X/B},f^*\omega_B)\otimes \mC(b)\cong H^1(X_b,T_{X_b})\otimes T_{B,b}^\vee\cong \textnormal{Ext}^1(\Omega^1_{X_b},\sO_{X_b})\otimes T_{B,b}^\vee.
	\end{equation*}
\end{lem}
\begin{proof}
	The five term exact sequence associated to the local to global spectral sequence recalled in Theorem \ref{proprieta} Point (5)
	\begin{equation*}
		\begin{split}
			0\to R^1f_*\hhom(\Omega^1_{X/B},f^*\omega_B)\to \ext^1_f(\Omega^1_{X/B},f^*\omega_B)\to f_*\ext^1(\Omega^1_{X/B},f^*\omega_B)\to\\\to R^2f_*\hhom(\Omega^1_{X/B},f^*\omega_B)\to \ext^2_f(\Omega^1_{X/B},f^*\omega_B)
		\end{split}		
	\end{equation*} gives the desired injection. Note that on $X^0=f^{-1}(B^0)$, $\Omega^1_{X/B}$ is locally free, hence $\ext^1(\Omega^1_{X/B},f^*\omega_B)$ is zero and this injection is an isomorphism on $B^0$:
	\begin{equation}
		\begin{split}
		 \ext^1_f(\Omega^1_{X/B},f^*\omega_B)_{|B^0}\cong R^1f_*\hhom(\Omega^1_{X/B},f^*\omega_B)\cong R^1f_*(T_{X/B})\otimes \omega_B.
		\end{split}
	\end{equation} The last statement is the Proper base change theorem \cite[Theorem 12.11]{H1}.
\end{proof}

We note that specializing the Global Kodaira-Spencer
\begin{equation}
	T_B\to\ext^1_f(\Omega^1_{X/B},\sO_X)
\end{equation}  in $b\in B'$ we  get the well known Kodaira-Spencer map at the point $b$
\begin{equation}
	\label{ks}
	T_{B,b}\to H^1(X_b,T_{X_b})\cong \textnormal{Ext}^1(\Omega^1_{X_b},\sO_{X_b}).
\end{equation}

By a famous general result, the Global Kodaira-Spencer morphism is zero on an open subset of $B$ if and only if the family is locally trivial on this set. In particular all the fibers are isomorphic and  the map (\ref{ks}) is zero in every point. Conversely it is not true that if (\ref{ks}) is zero in every point, then the Global Kodaira-Spencer is also zero. This holds however when the family is regular, i.e. the dimension of the complex vector space $H^1(X_b, T_{X_b})$ is the same for all points in the set. See for example \cite[Section 4]{K}.

\begin{rmk}
	To our knowledge $B^0= B'$ if the fibration is regular. In general the relation between  $B'$ and $B^0$ seems to be not fully clarified,
\end{rmk}


\begin{lem}
	\label{ss2}
	We have a surjective morphism 
	\begin{equation}
		\rho\colon \textnormal{Ext}^{1}(\Omega^1_{X/B},f^*\omega_B)\to H^0(B,\ext^1_f(\Omega^1_{X/B},f^*\omega_B))
	\end{equation} which is also an isomorphism if the general fiber of $f\colon X\to B$ is of general type.
Calling $\xi\in \textnormal{Ext}^{1}(\Omega^1_{X/B},f^*\omega_B)$ the element corresponding to Sequence (\ref{relatdiff}), $\rho$ maps $\xi$ to the Global Kodaira-Spencer map $\rho(\xi)$ which associates to  $b\in B'$ the element $\xi_b\in H^1(X_b,T_{X_b})\otimes T_{B,b}^\vee$ as defined in (\ref{xib}).
\end{lem}
\begin{proof}
	From the spectral sequence in Theorem \ref{proprieta} Point (6), we get the associated five terms exact sequence:
	\begin{equation*}
		\begin{split}
			0\to H^1(B,f_*\hhom(\Omega^1_{X/B},f^*\omega_B))\to\text{Ext}^{1}(\Omega^1_{X/B},f^*\omega_B)\stackrel{\rho}{\rightarrow} H^0(B,\ext^1_f(\Omega^1_{X/B},f^*\omega_B))\to\\\to H^2(B,f_*\hhom(\Omega^1_{X/B},f^*\omega_B))\to \text{Ext}^{2}(\Omega^1_{X/B},f^*\omega_B).
		\end{split}
	\end{equation*} The fourth term is zero because $B$ is a curve, hence $\rho$ is surjective.

Now note in the first term of this sequence that $$f_*\hhom(\Omega^1_{X/B},f^*\omega_B)=f_*\hhom(\Omega^1_{X/B},\sO_X)\otimes \omega_B=f_*T_{X/B}\otimes \omega_B.$$
Since $f_*T_{X/B}$ is torsion free, it is a line bundle on $B$ and if the general fiber of $f$ is of general type then $f_*T_{X/B}=0$ and we get the desired isomorphism.

 For the last statement, $\rho$ maps $\xi$ to a global section of $\ext^1_f(\Omega^1_{X/B},f^*\omega_B)$ which
  associates to the general $b\in B'$ the element $\xi_b\in H^1(X_b,T_{X_b})\otimes T_{B,b}^\vee$ as defined in (\ref{xib}); see \cite[Lemma 2.1]{lange}.
\end{proof} 
\begin{rmk}
	If the general fiber is of general type, the map $\rho$ is actually surjective even if $\dim B>1$.
\end{rmk}

\subsection{Global Kodaira-Spencer supported on a horizontal divisor}
Let $L\leq\Gamma(A,\mD^1)$ be a vector space of sections of the local system $\mD^1$ and choose $\eta_i$, $i=1,\dots, l$, forming a basis for $L$. Denote by $s_i$ the liftings of these sections via the splitting of (\ref{split}) fixed above.
We define the following divisors in $f^{-1}(A)$.

\begin{defn} 
Let $\sD^A$ be the divisor in $f^{-1}(A)$ given by the common zeroes of the sections $s_{i_1}\wedge\dots\wedge s_{i_{n-1}}\wedge \sigma$ where the $s_i$ run among the liftings above and $\sigma$ over the local sections of $\omega_B$ on $A$.

Denote by $\sD_{Hor}^A$ the divisor obtained by the horizontal components of $\sD^A$ and with  $D_b$ the restriction $\sD_{Hor}^A$ to the general fiber $X_b$. 

\end{defn}Note that $D_b$ is the fixed part of the sections $\eta_{i_1}\wedge\dots\wedge \eta_{i_{n-1}}$ where the $\eta_i$ run among the elements of the basis of $L$.
\begin{rmk}
\label{remark}
First note that $\sD^A$ and $\sD^A_{Hor}$ do not depend on the choice of the splitting of (\ref{split}) fixed above. In fact a different choice gives new liftings $\tilde{s_i}$, with $s_i-\tilde{s_i}\in \Gamma(A,\omega_B)$.

Furthermore consider  a Massey product $\omega=s_{i_1}\wedge\dots\wedge s_{i_{n}}$ of sections of $L$. By local computation it is clear that $\omega$ vanishes on $\sD^A_{Hor}$.
\end{rmk}

 We  can define the following sheaf on $A$: $$\ext^{1}_f(\Omega^1_{X/B}(-\sD^A_{Hor}),f^*\omega_B):=\ext^{1}_f({\Omega^1_{X/B}}_{|f^{-1}(A)}(-\sD^A_{Hor}),{f^*\omega_{B}}_{|A}).$$

Alternatively recall that by $\mL$ we denote the local system generated by $L$, $H_\mL$ the subgroup of $\pi_1(B, b)$ acting trivially on $\mL$ and $G_\mL=\pi_1(B, b)/H_\mL$ the monodromy group.

Let $\widetilde{B}\to B$ the covering classified by the subgroup $H_\mL$ and $\tilde{f}\colon\widetilde{X}\to \widetilde{B}$ the associated pullback fibration. The inverse image of the local system $\mL$ on $\widetilde{B}$ is trivial, in particular the sections $\eta_i$ are global and their liftings $s_i$ are global closed 1-forms on $\widetilde{X}$. This means that $\sD^A$ and $\sD^A_{Hor}$ define global divisors $\widetilde{\sD}$ and $\widetilde{\sD}_{Hor}$ on $\widetilde{X}$.
 Hence  $\ext^{1}_f(\Omega^1_{\widetilde{X}/\widetilde{B}}(-\widetilde{\sD}_{Hor}),f^*\omega_{\widetilde{B}})$ is defined on the whole base $\widetilde{B}$. 
 
\begin{rmk}\label{mono} When $\mL$ is Massey trivial generated and strict, by \cite[Theorem B]{RZ4} the monodromy of $\mL$ is finite hence the covering $\widetilde{B}\to B$ is also finite and  $\tilde{f}\colon\widetilde{X}\to \widetilde{B}$ is a fibration of projective varieties. So, under these hypotheses of Massey triviality, it is not restrictive to assume that everything is globally defined, since this is true up to a finite covering which does not impact the local deformation data of the fibers. Note that $\rho(\xi)$ is a global section of $\ext^1_f(\Omega^1_{X/B},f^*\omega_B)$ which defines a global section $\widetilde{\rho(\xi)}$ of  $\ext^1_{\tilde{f}}(\Omega^1_{{\widetilde{X}}/{\widetilde{B}}},\tilde{f}^*\omega_{\widetilde{B}})$; for example by Theorem \ref{proprieta} Point (2).
\end{rmk}
  
Finally we note that the relative $\ext$ functors are contravariant in the first component and we obtain a morphism (on $A$)
$$
\ext^1_f(\Omega^1_{X/B},f^*\omega_B)\to\ext^{1}_f(\Omega^1_{X/B}(-\sD^A_{Hor}),f^*\omega_B).
$$
\begin{rmk}
	By the same arguments seen in Lemma \ref{generalg}, we have that $$\ext^{1}_f(\Omega^1_{X/B}(-\sD^A_{Hor}),f^*\omega_B)\otimes \mC(b)\cong\textnormal{Ext}^1(\Omega^1_{X_b}(-D_b), \sO_{X_b})\otimes T_{B,b}^\vee$$ for general $b\in A$.
\end{rmk}

	We recall that $\xi_b\in H^1(X_b,T_{X_b})$ is \textit{supported} on a divisor $E_b$ in $X_b$ if 
		\begin{equation}
		\xi_b\in \Ker H^1(X_b,T_{X_b})\to H^1(X_b,T_{X_b}(E_b)).
	\end{equation}  See \cite{RZ1}.
The new concept of global supported deformation is Definition \ref{classe}, that we recall:
\begin{defn}	\label{supportato3}
	We say that $\rho(\xi)$ is supported on a horizontal divisor  $\sE$ in $f^{-1}(A)$ if 
	\begin{equation}
			\rho(\xi)_{|A}\in \Ker H^0(A,\ext^1_f(\Omega^1_{X/B},f^*\omega_B))\to H^0(A,\ext^{1}_f(\Omega^1_{X/B}(-\sE),f^*\omega_B)).	
\end{equation}
\end{defn}

By what we have seen so far, if $\rho(\xi)$ is supported on $\sD^A_{Hor}$ then  $\xi_b$ is supported on $D_b$ for the general $b\in B$. The viceversa is not true, since $\ext^{1}_f(\Omega^1_{X/B}(-\sD^A_{Hor}),f^*\omega_B)$ in general has a torsion part. 

Note also that if $\rho(\xi)$ is supported on $\sD^A_{Hor}$, we have that in the following diagram of torsion free sheaves on $f^{-1}(A)$
\begin{equation}
	\label{qui}
	\xymatrix{0\ar[r]& f^*\omega_B\ar[r]\ar@{=}[d]&{\sE}\ar[r]\ar[d]&\Omega^1_{X/B}(-\sD^A_{Hor})\ar[r]\ar[d]&0\\
		0\ar[r]& \ar[r]f^*\omega_B&\Omega^1_X \ar[r]&\Omega^1_{X/B}\ar[r]&0
	}
\end{equation} the top row splits when restricted to the general fiber. Of course this does not mean that the top row itself splits.

\subsection{Global supported deformations and Massey triviality}
In light of the Adjoint theorem \cite[Theorem A]{RZ1} we have the following result

\begin{thm}
	\label{aggiuntarel}
Let $L\leq\Gamma(A,\mD^1)$ be a Massey trivial vector space. Assume that $L$  generically generates $\Omega^1_{X_b}$ on the general fiber. Then $\rho(\xi)$ is supported on $\sD^{A'}_{Hor}$, where $A'\subset A$ is a suitable open dense subset. Furthermore if  $\ext^{1}_f(\Omega^1_{{{X}}/{{B}}}(-\sD^{A}_{Hor}),{f}^*\omega_{{B}})$ is torsion free, then $\rho(\xi)$ is supported on $\sD^{A}_{Hor}$. 
\end{thm}
\begin{proof}
	From the hypotheses on $L$ we know that the monodromy is finite and we can work on the covering $\widetilde{X}$ which is projective. Here everything is defined globally, see Remark \ref{mono}. 
	
	Now choose generic $\eta_{i_1}, \dots, \eta_{i_n}$ linearly independent elements of $L$. They are Massey trivial by hypothesis hence by the Adjoint Theorem  \cite[Theorem A]{RZ1} we have that on a smooth fiber $X_b$ the infinitesimal deformation $\xi_b$ is supported on a divisor $D_b^{i_1,\dots,i_n}$, defined as the fixed part of the $n$ sections $\eta_{i_1}\wedge \dots\wedge\widehat{\eta_{i_j}}\wedge\dots\wedge \eta_{i_n}$. By \cite[Proposition 3.1.6]{PZ}, if $L$ generically generates $\Omega^1_{X_b}$, it turns out that actually  $D_b^{i_1,\dots,i_n}$ does not depend on the choice of the $\eta_i$ and it is exactly the divisor $D_b$.
	
	We have proved  that $\xi_b$ is supported on $D_b$ which of course is the restriction of ${\sD_{Hor}}$ on the fiber $X_b$. The thesis follows easily.
\end{proof}

\begin{rmk}
\label{remagg}
One could also add the strictness hypothesis and state the theorem globally over $\widetilde{B}$.
\end{rmk}
Recall that we say that  $L\leq H^0(X_b, \Omega^1_{X_b})$ is strongly strict if the map $\bigwedge^{n-1} L\to H^0(X_b,\omega_{X_b})$ is an isomorphism on the image.
\begin{cor}
Let $f\colon X\to B$ be a family such that the general fiber $X_b$ is a variety of general type with $p_g(X_b)=\rank L=n$. If $L$ is strongly strict and $\Omega^1_{X_b}$ is generated by the elements of $L$, then $f$ is isotrivial on an appropriate dense open set of the base.
\end{cor}
\begin{proof}
For such an $X_b$ it is not difficult to see that $L$ is Massey trivial and  $D_b=0$, see for example \cite[Corollary 2.2.2]{RZ1}. The idea is that if we take a basis $\eta_1,\dots,\eta_n$ of $L$, $\bigwedge^{n-1} L\cong H^0(X_b,\omega_{X_b})$ and this implies that the $\eta_i$ are necessarily Massey trivial.
 $D_b=0$ since $L$ generates $\Omega^1_{X_b}$.

 Hence by Theorem \ref{aggiuntarel} we have that $\rho(\xi)$ is supported on an empty divisor, that is $\rho(\xi)$ is trivial.
 
 This means that the fibration is isotrivial on an appropriate open set of the base.
\end{proof}

Another straightforward application of Theorem \ref{reladjquadthm} gives the following result which finally relates the notion of relative adjoint quadric to the notion of being supported on ${\sD_{Hor}}$.
\begin{cor}
	Let $L\leq\Gamma(A,\mD^1)$ be a vector space. Assume that $L$ generically generates $\Omega^1_{X_b}$ on the general fiber and that there are no relative adjoint quadrics. 
	Then $\rho(\xi)$ is supported on $\sD^{A'}_{Hor}$. Furthermore if  $\ext^{1}_f(\Omega^1_{{{X}}/{{B}}}(-\sD^{A}_{Hor}),{f}^*\omega_{{B}})$ is torsion free, then $\rho(\xi)$ is supported on $\sD^{A}_{Hor}$. 
%
\end{cor}

Finally we prove a viceversa of Theorem \ref{aggiuntarel}. These two results together are Theorem \ref{A}.
\begin{thm}
	\label{viceversa}
Assume that $L\leq\Gamma(A,\mD^1)$ is a vector space and $\rho(\xi)$ is supported on ${\sD^A_{Hor}}$. If $f_*\sO_X(\sD^A_{Hor})$ is a line bundle then the vector space $L$ is Massey trivial.
\end{thm}


\begin{proof}
We want to prove that $L$ is Massey trivial, that is every choice of $n$ linearly independent sections in $L$ is Massey trivial. We fix such a choice $\eta_1,\dots,\eta_n$ and start by considering the exact sequence 
\begin{equation}
\label{rel2}
	0\to f^*\omega_B\to \Omega^1_{X}\to\Omega^1_{X/B}\to0
\end{equation} and applying the functor $f_*\hhom(\cdot,f^*\omega_B)$ to obtain the long exact sequence 
\begin{equation}
	0\to f_*T_{X/B}\otimes \omega_B\to f_*T_{X}\otimes \omega_B\to \sO_B\to \ext^1_f(\Omega^1_{X/B},f^*\omega_B)\to \dots	
\end{equation} Recall that the image of $1\in H^0(B,\sO_B)$ is $\rho(\xi)\in H^0(B,\ext^1_f(\Omega^1_{X/B},f^*\omega_B))$.

We ask the reader to accept the following easier notation for the rest of this proof: $X:=f^{-1}(A)$ and $B:=A$, that is we restrict locally on $A$.

Sequence (\ref{rel2}) together with its tensor by $\sO_X(-\sD^A_{Hor})$ fits into the commutative diagram 
\begin{equation}
	\xymatrix{0\ar[r]&f^*\omega_B\ar[r]\ar@{=}[d]&\Omega^1_X\ar[r]&\Omega^1_{X/B}\ar[r]&0\\
		0\ar[r]&f^*\omega_B\ar[r]&\sE\ar[r]\ar[u]&\Omega^1_{X/B}(-\sD^A_{Hor})\ar@{=}[d]\ar[r]\ar[u]&0\\
0\ar[r]&f^*\omega_B(-\sD^A_{Hor})\ar[r]\ar[u]&\Omega^1_X(-\sD^A_{Hor})\ar[r]\ar[u]&\Omega^1_{X/B}(-\sD^A_{Hor})\ar[r]&0}
\end{equation}
Taking again the $f_*\hhom(\cdot,f^*\omega_B)$ we obtain 
\begin{equation}
	\xymatrix{0\ar[r]&f_*T_{X/B}\otimes \omega_B\ar[r]\ar[d]&f_*T_{X}\otimes \omega_B\ar[r]\ar[d]&\sO_B\ar@{=}[d]\ar[r]&\ext^1_f(\Omega^1_{X/B},f^*\omega_B)\ar[d]\\
		0\ar[r]&f_*T_{X/B}(\sD^A_{Hor})\otimes \omega_B\ar[r]\ar@{=}[d]&f_*\sE^\vee\otimes \omega_B\ar[r]\ar[d]&\sO_B\ar[r]\ar[d]&\ext^1_f(\Omega^1_{X/B}(-\sD^A_{Hor}),f^*\omega_B)\ar@{=}[d]\\
0\ar[r]&f_*T_{X/B}(\sD^A_{Hor})\otimes \omega_B\ar[r]&f_*T_{X}(\sD^A_{Hor})\otimes \omega_B\ar[r]&f_*\sO_X(\sD^A_{Hor})\ar[r]&\ext^1_f(\Omega^1_{X/B}(-\sD^A_{Hor}),f^*\omega_B)	}
\end{equation}
Thanks to this diagram we can interpret our hypothesis that $\rho(\xi)$ is supported on $\sD^A_{Hor}$ as follows.

As pointed out above, the identity element in the first row goes to $\rho(\xi)$ in $\ext^1_f(\Omega^1_{X/B},f^*\omega_B)$. By hypothesis $\rho(\xi)$ goes to zero in $\ext^1_f(\Omega^1_{X/B}(-\sD^A_{Hor}),f^*\omega_B)$, hence the identity element in the second row is in the image of the morphism $f_*\sE^\vee\otimes\omega_B\to \sO_B$. This means that if we take a point $b\in B$, we can, locally around $b$, find a lifting of the identity in $f_*\sE^\vee\otimes\omega_B$. We denote by $\theta_b$ the image of this local lifting in $f_*T_{X}(\sD^A_{Hor})\otimes \omega_B$.

Denote by $\bigwedge^{n-1}W$ the vector space with basis the sections $\omega_i=s_1\wedge\dots\wedge\widehat{s_i}\wedge \dots\wedge s_n$ as in Definition \ref{omegai} and consider the following commutative square
\begin{equation}
	\xymatrix{
		f_*T_{X}(\sD^A_{Hor})\otimes \omega_B\ar^{\alpha}[r]\ar^-{\alpha'}[d]&f_*\sO_X(\sD^A_{Hor})\ar^{\beta}[d]\\
	\bigwedge^{n-1}W\otimes f_*\sO_X(\sD^A_{Hor})\otimes\omega_B\ar^-{\beta'}[r]&f_*\omega_X	}
\end{equation}	The horizontal arrow $\alpha$ is the same as in the above diagram, and the  horizontal arrow $\beta'$ is given by the fact that the $\omega_i$ are elements of $f_*\Omega^{n-1}_X$ and furthermore $\sD^A_{Hor}$ is a divisor of common zeroes of $\omega_i\wedge \sigma$ for arbitrary $\sigma$ in $\omega_B$, that is we can see $\omega_i\wedge \sigma\in \bigwedge^{n-1}W\otimes\omega_B$ as an element of $f_*\omega_X(-\sD^A_{Hor})$.

The vertical arrow $\alpha'$ is given by taking a section $\theta$ of $f_*T_{X}(\sD^A_{Hor})\otimes \omega_B$ and sending it to
$$
\theta\mapsto \sum_i (-1)^i \theta(s_i)\otimes \omega_i
$$ where $\theta(s_i)$ indicates the contraction, since  $s_i$ is in $f_*\Omega^1_X$.

The vertical arrow $\beta$ is given by the Massey product $\omega$ since we recall  that $\omega$ vanishes on $\sD^A_{Hor}$, see Remark \ref{remark}.

With these definitions, it is not difficult to see that the square  commutes, hence $\beta\alpha(\theta_b)=\beta'\alpha'(\theta_b)$. On one side $\beta\alpha(\theta_b)=\beta(1)=\omega$ since $\theta_b$ is a lifting of the identity.

On the other side, note that we are working locally around the general point $b$. The germ of the section $\theta_b$ can be decomposed as a sum of elements of the form
$v_b\otimes \sigma_b$ with $v_b\in H^0(X_b,T_{X_b}(D_b))$ and $\sigma_b\in \omega_{B,b}$. Its image via $\alpha'$ is then a sum of sections $v_b(s_i)\otimes \sigma_b\otimes \omega_i$ where now $v_b(s_i)\in H^0(X_b,\sO_{X_b}(D_b))$. Since by hypothesis $f_*\sO_X(\sD^A_{Hor})$ is a line bundle, we have that $h^0(X_b,\sO_{X_b}(D_b))=1$. This implies that the poles of $v_b(s_i)$ are exactly the zeroes of $\omega_i\otimes \sigma_b$, hence the image via $\beta'$ is exactly an element of $\langle\omega_i\rangle\otimes \omega_B$.

Hence by the commutativity we conclude that the Massey product $\omega$ is in the submodule $\langle\omega_i\rangle\otimes \omega_B$, that is it is Massey trivial by Definition \ref{mtrivial}. 
\end{proof}

\subsection{Global supported deformations and morphisms to product varieties}
Now recall the Generalized Castelnuovo-de Franchis theorem, see \cite[Theorem 1.14]{Ca2} and \cite[Prop II.1]{Ran}. See also \cite[Theorem 5.6]{RZ4} for the following refined version.
\begin{thm}
	\label{cas2}
	Let $Z$ be an $n$-dimensional compact K\"ahler manifold and $w_1,\dots, w_l \in H^0 (Z,\Omega^1_Z)$ linearly independent 1-forms such that $w_{j_1}\wedge\dots\wedge w_{j_{k+1}}= 0$ for every $j_1,\dots,j_{k+1}$ and that no collection of $k$ linearly independent forms in the span of $w_1,\dots, w_{j_{k+1}}$ wedges to zero. Then there exists a holomorphic map $f\colon Z\to Y$ over a normal variety $Y$ of dimension $\dim Y=k$ and such that $w_i\in f^*H^0(Y,\Omega^1_Y)$. Furthermore $Y$ is of general type.
\end{thm}
We are now ready to prove Theorem \ref{B} which follows from the following corollary of Theorem \ref{viceversa}, and a nice interpretation of Theorem \ref{cas2}.
\begin{cor}
	\label{finito1}
	Consider $L\leq\Gamma(A, \mD^1)$ a strict vector subspace.
	Assume that $\rho(\xi)$ is supported on ${\sD^A_{Hor}}$ and that $f_*\sO_X(\sD^A_{Hor})$ is a line bundle. Then  for every $U\subseteq B$ open subset trivializing the local system $\mL$, there exists a fibration $h_U\colon f^{-1}(U)\to Y$ to a normal $n-1$ dimensional variety $Y$ of general type. Furthermore, up to a finite covering $\widetilde{X}\to X$, all the maps $h_U$ glue together to a map $h\colon \widetilde{X}\to Y$. 
\end{cor}
\begin{proof}
	By Theorem \ref{viceversa}, we know that $L$ is Massey trivial. In particular by Proposition \ref{localglobal}, there exists a unique lifting $\widetilde{L}$ such that the wedge 
	$$
	\bigwedge^n\widetilde{L}\to \Gamma(A,f_*\omega_X)
	$$ is zero. We recall that the sections in $\widetilde{L}$ can be seen as 1-forms in $\Omega^1_{X,d}$.
	
	Since their wedge is zero, we can then apply Theorem \ref{cas2} and this give a morphism with connected fibers $h_A\colon f^{-1}(A)\to Y$ onto a normal $n-1$ dimensional variety $Y$. Note that even if $f^{-1}(A)$ is not compact,  Theorem \ref{cas2} can still be applied because the sections in $\widetilde{L}$ are \textit{closed} by Proposition \ref{localglobal}. This is actually enough to ensure that the arguments of Theorem \ref{cas2} applies. In particular see \cite[Remark 5.7]{RZ4}.
	
	Given another open subset $U\subset B$ as in the statement, the proof is similar and relies on the fact that Proposition \ref{localglobal} basically allows to pass from a local to a global condition. More precisely, by \cite[Theorem B]{RZ4} the monodromy group $G_\mL$ is finite, hence the covering associated is also finite. Furthermore the sections of $L$ give global sections in $\widetilde{X}$. So by Proposition \ref{localglobal} applied on $\widetilde{B}$, the elements of $\widetilde{L}$ can be seen as global closed 1-forms on $\widetilde{X}$ such that 
		$$
	\bigwedge^n\widetilde{L}\to \Gamma(\widetilde{B},\tilde{f}_*\omega_{\widetilde{X}})
	$$ is zero. This is a global condition deriving from the local  condition of Massey triviality.
	
	Hence on $\widetilde{X}$ we can work globally and this gives the morphisms $h_U$ which glue together in a morphism $h$ on $\widetilde{X}$.
	
	Of course note that the $h_U$ in general are not unique due to monodromy, hence they glue together to give $h\colon \widetilde{X}\to Y$ but they may not glue to give a morphism $X\to Y$.
\end{proof}
This Corollary  is Theorem \ref{B}.
\begin{cor}
	\label{finito2}
	Let $L$ be a strict vector subspace as above. Assume that $\rho(\xi)$ is supported on ${\sD^A_{Hor}}$ and that $f_*\sO_X(\sD^A_{Hor})$ is a line bundle. Then there exists a generically finite surjective morphism $\widetilde{X}\to\widetilde{B}\times Y$. 
\end{cor}
\begin{proof}
	Note that the map $h$  from the previous corollary is surjective when restricted to the general fiber $X_b$ thanks to the strictness hypothesis. Hence the map $\tilde{f}\times h\colon \widetilde{X}\to \widetilde{B}\times Y$ is generically finite.
\end{proof}

\begin{rmk}
Consider the morphism $h\colon \widetilde{X}\to Y$ and $X_b$ a general fiber of $\tilde{f}$ over $b\in \widetilde{B}$. If the ramification of $h|_{X_b}$, denoted by $R_b$, is the restriction to $X_b$ of a divisor $\sR$ on $\widetilde{X}$ contained in the critical locus of $h$, then the deformation $\xi_b$ is trivial.

In fact in this case the pullback $h^*\omega_Y$ is not only contained in $\Omega^{n-1}_{\widetilde{X}}$ but also in $\Omega^{n-1}_{\widetilde{X}}(-\sR)$ and on the fiber $X_b$ we have the diagram
\begin{equation}\label{esattasplit}
	\xymatrix{
	0\ar[r]&\Omega^{n-2}_{X_b}(-R_b)\ar[r]&\Omega^{n-1}_{\widetilde{X}}{|_{X_b}}(-R_b)\ar[r]&\omega_{X_b}(-R_b)\ar[r]&0\\
&&(h^*\omega_Y{|_{X_b}})^{\vee\vee}\ar@{^{(}->}[u]\ar[ur]}
\end{equation}
The diagonal arrow is an isomorphism.
This gives the splitting of the exact sequence in (\ref{esattasplit}) which we note is associated to $\xi_b$ after tensoring by $R_b$.
\end{rmk}

\section{The case of volume forms on the fibers}
\label{sez5}

In this section we study the case of $n-1$ forms.
Note that there is a map given by taking the wedge exact sequence  
$$
\textnormal{Ext}^1(\Omega^1_{X/B},f^*\omega_B)\to \textnormal{Ext}^1(\Omega^{p}_{X/B},f^*\omega_B\otimes\Omega^{p-1}_{X/B} )
$$ 
but since our fibration $f\colon X\to B$ has in general singular fibers, this map is not necessarily an isomorphism.
By the same arguments of Lemma \ref{ss2}   we have a map
$$
\textnormal{Ext}^1(\Omega^{p}_{X/B},f^*\omega_B\otimes\Omega^{p-1}_{X/B} )\to H^0(B,\ext^1_f(\Omega^{p}_{X/B},f^*\omega_B\otimes\Omega^{p-1}_{X/B} ) ).
$$

It make sense to consider the image of $\rho(\xi)$ in $H^0(B,\ext^1_f(\Omega^{p}_{X/B},f^*\omega_B\otimes\Omega^{p-1}_{X/B} ) )$. We give the following definition
\begin{defn}
	We say that $\rho(\xi)$ is $p$-supported on a horizontal divisor $\sE$ in $f^{-1}(A)$  if 
	\begin{equation}
		\rho(\xi)\in \Ker H^0(A,\ext^1_f(\Omega^1_{X/B},f^*\omega_B))\to H^0(A,\ext^1_f(\Omega^{p-1}_{X/B}(-\sE),f^*\omega_B\otimes\Omega^{p-1}_{X/B} )).
	\end{equation}
\end{defn}

By Theorem \ref{proprieta} Point (2), it is also easy to see that we have a homomorphism
$$
H^0(A,\ext^1_f(\Omega^1_{X/B}(-\sE),f^*\omega_B))\to H^0(A,\ext^1_f((\Omega^{p}_{X/B}(-\sE),f^*\omega_B\otimes\Omega^{p-1}_{X/B} )) 
$$
hence being supported on $\sE$ according to Definition \ref{supportato3} implies being $p$-supported on $\sE$, but the converse may not be true.
Furthermore the previous results, as Theorem \ref{viceversa} and its corollaries, work only in the case where the general fiber is highly irregular, that is $h^0(X_b,\Omega^1_{X_b})\geq n$. In this section we show  similar results which consider the property of being $n-1$-supported and also work with top forms instead of $1$-forms. This result will use the analogue of the Castelnuovo-de Franchis theorem \ref{cas2} in the case of $n-1$-forms, that is \cite[Theorem 7.2]{RZ5}. For the reader benefit and for final use we recall the set up. 

Let $Z$ be a smooth variety $w_1,\dots, w_l \in H^0 (Z,\Omega^p_Z)$, $l\geq p+1$, be linearly independent $p$-forms such that $w_i\wedge w_j=0$ (as an element of $\bigwedge^2\Omega^p_Z$ and not of $\Omega_Z^{2p}$) for any choice of $i,j=1,\dots, l$. These forms generate a subsheaf of $\Omega^p_Z$ generically of rank $1$. Note that the quotients $w_i/w_j$ define a non-trivial global meromorphic function on $Z$ for every $i\neq j$, $i,j=1,\dots, l$. By taking the differential $d (w_i/w_j)$ we then get global meromorphic $1$-forms on $Z$. We assume that there exist $p$ of these meromorphic differential forms $d (w_i/w_j)$ that do not wedge to zero; if this is the case we call  the subset $\{w_1,\dots, w_l \}\subset H^0 (Z,\Omega^p_Z)$ $p$-strict. For this setting, this condition is  analogous to the strictness condition considered in Definition \ref{strict}.

We need Theorem 7.2 in \cite{RZ5}:
\begin{thm}
	\label{pcast}
	Let $Z$ be an $n$-dimensional smooth variety and let $\{w_1,\dots, w_l \}\subset H^0 (Z,\Omega^p_Z)$ be a $p$-strict set. Then there exists a rational map $f\colon Z\dashrightarrow Y$ over a $p$-dimensional smooth variety $Y$ of general type such that the $w_i$ are pullback of some holomorphic $p$-forms $\mu_i$ on $Y$, that is $w_i=f^*\mu_i$, where $i=1,\dots , l$.
\end{thm}

Now consider $L\leq\Gamma(A,\mD^{n-1})$ a vector space of $n-1$ forms and assume that the elements of $L$ are liftable to $f_*\Omega_{X,d}^{n-1}$ in the sequence
 \begin{equation}
	0\to \omega_B\otimes f_*\Omega^{n-2}_{X/B}\to f_*\Omega^{n-1}_{X,d}\to \mD^{n-1}\to0.
\end{equation} 
This liftability assumption is needed by Remark \ref{solle}. In analogy with the previous case, we will denote by $\eta_i$ the elements of a basis of $L$ and by $s_i$ a fixed choice of liftings.   We will also denote by $\sD^{A,(n-1)}_{Hor}$ the horizontal part of the divisor given by the common zeroes of $s_i\wedge \sigma$ where $\sigma$ is a section of $\omega_{B}.$ 

We can finally prove the following result, which immediately gives Theorem \ref{C}.
\begin{thm}
	\label{viceversap}
	Assume that $\rho(\xi)$ is $n-1$-supported on $\sD^{A,(n-1)}_{Hor}$ and that the $\eta_i$ are $n-1$-strict.
	If $f_*\hhom( \Omega^{n-1}_{X/B}(-\sD^{A,(n-1)}_{Hor}),\Omega^{n-2}_{X/B})$ is zero then there exists a meromorphic dominant map $f^{-1}(A)\dashrightarrow
	 Y$ to a smooth $n-1$ dimensional variety $Y$ of general type. 
\end{thm}
\begin{proof}
Again as in Theorem \ref{viceversa}, we work on $A$ and we ask the reader to accept the simplification $X=f^{-1}(A)$. We also denote $\sD^{A,(n-1)}_{Hor}$ by simply by $\sD^{A}_{Hor}$ for convenience.

	Consider the exact sequence
$$
0\to f^*\omega_{{B}}\otimes \Omega^{n-2}_{{X}/{B}}\to \Omega^{n-1}_{{X}}\to \Omega^{n-1}_{{X}/{B}}\to 0,
$$ the $n-1$-th wedge of Sequence \ref{rel2}. This sequence together with its tensor by $\sO_{{X}}(-\sD^A_{Hor})$ fits into the commutative diagram 
\begin{equation}
	\xymatrix{0\ar[r]&f^*\omega_B\otimes \Omega^{n-2}_{X/B}\ar[r]\ar@{=}[d]&\Omega^{n-1}_X\ar[r]&\Omega^{n-1}_{X/B}\ar[r]&0\\
		0\ar[r]&f^*\omega_B\otimes \Omega^{n-2}_{X/B}\ar[r]&\sE\ar[r]\ar[u]&\Omega^{n-1}_{X/B}(-\sD^A_{Hor})\ar@{=}[d]\ar[r]\ar[u]&0\\
		0\ar[r]&f^*\omega_B\otimes \Omega^{n-2}_{X/B}(-\sD^A_{Hor})\ar[r]\ar[u]&\Omega^{n-1}_X(-\sD^A_{Hor})\ar[r]\ar[u]&\Omega^{n-1}_{X/B}(-\sD^A_{Hor})\ar[r]&0}
\end{equation}
We apply the functor $f_*\hhom(\cdot,f^*\omega_B\otimes \Omega^{n-2}_{X/B})$ and obtain the diagram
\begin{scriptsize}
\begin{equation}
	\xymatrix{&f_*\hhom(\Omega^{n-1}_X,\Omega^{n-2}_{X/B})\otimes \omega_B\ar[r]\ar[d]&f_*\hhom( \Omega^{n-2}_{X/B}, \Omega^{n-2}_{X/B})\ar@{=}[d]\ar[r]&\ext^1_f(\Omega^{n-1}_{X/B},f^*\omega_B\otimes\Omega^{n-1}_{X/B} )\ar[d]\\
		&f_*\hhom(\sE,\Omega^{n-2}_{X/B})\otimes \omega_B\ar[r]\ar[d]&f_*\hhom( \Omega^{n-2}_{X/B}, \Omega^{n-2}_{X/B})\ar[r]\ar[d]&\ext^1_f(\Omega^{n-1}_{X/B}(-\sD^A_{Hor}),f^*\omega_B\otimes\Omega^{n-1}_{X/B} )\ar@{=}[d]\\
	0\ar[r]&	f_*\hhom(\Omega^{n-1}_X(-\sD^A_{Hor}),\Omega^{n-2}_{X/B})\otimes \omega_B\ar^-\alpha[r]&f_*\hhom( \Omega^{n-2}_{X/B}(-\sD^A_{Hor}), \Omega^{n-2}_{X/B})\ar[r]&\ext^1_f(\Omega^{n-1}_{X/B}(-\sD^A_{Hor}),f^*\omega_B\otimes\Omega^{n-1}_{X/B} )	}
\end{equation}
\end{scriptsize}

Note that in the middle sheaf of the top row, that is $f_*\hhom( \Omega^{n-2}_{X/B}, \Omega^{n-2}_{X/B})$, we have the identity element that we simply denote by $1$, its image in $f_*\hhom( \Omega^{n-2}_{X/B}(-\sD^A_{Hor}), \Omega^{n-2}_{X/B})$ is the natural injection and we will denote it by $i$. Exactly as in the proof of Theorem \ref{viceversa}, by the hypothesis on $\rho(\xi)$, the section $i$ is in the image of $\alpha$ and can be locally lifted to  $f_*\hhom(\Omega^{n-1}_X(-\sD^A_{Hor}),\Omega^{n-2}_{X/B})\otimes \omega_B$. Actually by our hypothesis on the vanishing of $f_*\hhom( \Omega^{n-1}_{X/B}(-\sD^{A}_{Hor}),\Omega^{n-2}_{X/B})=\ker\alpha$, this lifting is global (on $A$) and unique. We will denote it by $h$.

Now fix $\eta_1,\eta_2$ two sections of $L$ and $s_1,s_2$ the associated liftings. 
It is not difficult to see by a local computation that 
$$
s_1\wedge s_2+h(s_1)\wedge h(s_2)=s_1\wedge h(s_2)-s_2\wedge h(s_1).
$$
Hence by changing the liftings of the $\eta_i$ to $\tilde{s}_i=s_i-h(s_i)$ we have that $\tilde{s}_1\wedge\tilde{s}_2=0$. Since $h$ is unique and works for every pair of sections, repeating the same argument, we get that $\tilde{s}_i\wedge\tilde{s}_j=0$ for any pair of sections $\eta_i,\eta_j$ of $L$. Hence we can apply Theorem \ref{pcast} to get the map $f^{-1}(A)\dashrightarrow
	 Y$. As in Corollary \ref{finito1}, even if $f^{-1}(A)$ is not compact, the argument is the same since the forms $s_i$ are closed.
\end{proof}

In analogy with Corollary  \ref{finito2}, we want to study the situation on a suitable covering $\widetilde{X}$ of $X$. In particular, in analogy with Section \ref{sez4}, we note that up to appropriate covering the $\eta_i$ and $s_i$ are global and define a global horizontal divisor $\widetilde{D}_{Hor}$ on $\widetilde{X}$. The following corollary shows that this covering is finite.
\begin{cor}
	\label{globale}
	Under the same assumptions of the previous theorem, the covering  $\widetilde{X}\to X$ which trivializes the local system $\mL$ generated by $L$ is finite. On this covering assume that $\ext^1_{\tilde{f}}(\Omega^{n-1}_{\widetilde{X}/\widetilde{B}}(-\widetilde{D}_{Hor}),\tilde{f}^*\omega_{\widetilde{B}}\otimes\Omega^{n-1}_{\widetilde{X}/\widetilde{B}} )$ is torsion free or alternatively that $\widetilde{\rho(\xi)}$ is $n-1$-supported on $\widetilde{D}_{Hor}$. Then we  have a rational dominant map $h\colon \widetilde{X}\dashrightarrow Y$.
\end{cor}
\begin{proof}
As in Theorem \ref{viceversap}, we can find  liftings of the $\eta_i$ with  $\tilde{s}_i\wedge\tilde{s}_j=0$. By \cite[Theorem 5.10]{R3}, the monodromy of the local system $\mL$ generated by $L$ is finite, hence as in Corollary \ref{finito1}, the sections $\eta_i$ and $s_i$ are global on a finite covering $\widetilde{X}$ of $X$ and here we can work globally. 

The second part of the statement follows in the same fashion of Theorem \ref{viceversap} once we note that the two alternative hypotheses ensure the existence of $h$ as in the proof of the previous theorem.
\end{proof}
\begin{cor}
	Under the same assumptions of the previous Corollary, we have a generically finite rational map $\widetilde{X}\dashrightarrow \widetilde{B}\times Y$.
\end{cor}
\begin{proof}
	The rational map $\widetilde{X}\dashrightarrow \widetilde{B}\times Y$ is of course given by $\tilde{f}\times h$.
	
\end{proof}

\begin{rmk}
	Note that if $p_g(X_b)\geq a\cdot p_g(Y)-b,$ with $a>0$ and $b\in \mZ$, the rank $r$ of the local system $\mL$ is
	$$
	r\leq \frac{p_g(X_b)+b}{a}.
	$$
\end{rmk}

\subsection{Fibered threefolds}
We can obtain some bounds for the geometric genus $p_g(Y)$, where $Y$ is as in Theorem \ref{viceversap}, in  the case of a relatively minimal fibered threefold $f\colon X\to B$. These bounds are based on the work of \cite{Ba} and \cite{Ri}.

We start by recalling some standard definitions. Let $f_*\omega_{X/B}=\sU\oplus\sA$ the second Fujita decomposition of the direct image of the  relative dualizing sheaf and $u_f:=\rank \sU$ the rank of the unitary flat part, so that $\rank\sA=p_g(X_b)-u_f$.
Denote also by $g$ the genus of $B$ and by $K_f$ the divisor of $\omega_{X/B}$, we have the following invariants for our fibration
$$
K_f^3:=K_X^3-2(g-1)K^2_{X_b},
$$ 
$$
\Delta_f:=deg f_*\sO_X(K_f),
$$ $$ \chi_f:=\chi_{X_b}\chi_B-\chi_X.
$$
Hence, for fibered threefolds, it makes sense to define two slopes
$$
\lambda^1_f:=\frac{K_f^3}{\chi_f}, \quad \lambda^2_f:=\frac{K_f^3}{\Delta_f}.
$$
From now on we assume that $\chi_f> 0$  so that, following \cite[Lemma 5.6]{Ba}, we have that $\chi_f\leq\Delta_f$ and hence more importantly $\lambda^2_f\leq \lambda^1_f$. We refer to \cite[Theorem 5.7]{Ba} for examples where $\chi_f\geq 0$.

We also use the following definition from \cite{Ri}.
\begin{defn}
Let $|M|$ be a linear system on a surface $S$. We say that 
\begin{itemize}
	\item $|M|$ is g.f.d. if it induces a generically finite map $\phi_{|M|}\colon S\to \mP^k$ which is a double cover on the image which is a ruled surface
	\item $|M|$ is g.f.n.d. if it induces a generically finite map which is not a double cover on a ruled surface
	\item $|M|$ is a fibration of gonality $\gamma$ if $\phi_{|M|}\colon S\to \mP^k$ is a fibration with general fiber a smooth curve of gonality $\gamma$
\end{itemize} 
\end{defn}	

The key point of the estimates of this section is the assumption that the ample part of the Fujita decomposition $\sA$ is semistable. Under this assumption we have that 
\begin{equation}
	0\subsetneq \sA\subsetneq f_*\omega_{X/B}
\end{equation} is the Harder-Narasimhan filtration of $f_*\omega_{X/B}$ and we denote by $|M_\sA|$ the movable part of $\sA$ restricted to the fiber $X_b$.

We have the following diagram which puts together all the relevant pieces for our purposes.
\begin{equation}
	\xymatrix{&Y\ar@{-->}[rr]^{\phi_{|K_Y|}}&&\mP^{p_g(Y)-1}\\
		X_b\ar@{-->}[rr]^{\phi_{|K_{X_b}|}}\ar@{-->}[ru]^h\ar@{-->}[rrrd]_{\phi_{|M_\sA|}}&&\mP^{p_g(X_b)-1}\ar@{-->}[ru]\ar@{-->}[rd]&\\
		&&&\mP^k}
\end{equation}
Note that the diagonal maps between the projective spaces are just projections.

The result is the following
\begin{prop}
	Let $f\colon X\to B$ be a relatively minimal fibered threefold with $\chi_f >0$ and $g\leq1$. Assume that the ample part of the Fujita decomposition $\sA$ is semistable.  Under the hypotheses of Theorem \ref{viceversap} we have the following bound on $p_g(Y)$.
	
	If $\rank \sA\geq 2$
	\begin{itemize}
		\item If $|K_{X_b}|$ and $|M_{\sA}|$ are g.f.n.d.
		$$
		p_g(Y)\leq \frac{63p_g(X_b)+20}{66}
		$$
		\item If $|K_{X_b}|$ is g.f.n.d. and $|M_{\sA}|$ is g.f.d. 
		$$
		p_g(Y)\leq \frac{65p_g(X_b)-4q+14}{68}
		$$
		\item If $|K_{X_b}|$ is g.f.n.d. and $|M_{\sA}|$ defines a fibration of gonality $\gamma\geq 5$
		$$
		p_g(Y)\leq \frac{64p_g(X_b)+12}{67}
		$$
		\item If $|K_{X_b}|$ is g.f.n.d. and $|M_{\sA}|$ defines a fibration of gonality $\gamma\geq 4$
		$$
		p_g(Y)\leq \frac{65p_g(X_b)+11}{68}
		$$
		\item If $|K_{X_b}|$ is g.f.n.d. and $|M_{\sA}|$ defines a fibration of gonality $\gamma\geq 3$
		$$
		p_g(Y)\leq \frac{67p_g(X_b)+10}{70}
		$$
		\item If $|K_{X_b}|$ is g.f.n.d. and $|M_{\sA}|$ defines a fibration of gonality $\gamma\geq 2$
		$$
		p_g(Y)\leq \frac{67p_g(X_b)+9}{70}
		$$
		\item If $|K_{X_b}|$ and $|M_{\sA}|$ are g.f.d.
		$$
		p_g(Y)\leq \frac{66p_g(X_b)-6q+11}{68}
		$$
		\item If $|K_{X_b}|$ is g.f.d. and $|M_{\sA}|$ defines a fibration of gonality $\gamma\geq 5$
		$$
		p_g(Y)\leq \frac{65p_g(X_b)-2q+9}{67}
		$$
		\item If $|K_{X_b}|$ is g.f.d. and $|M_{\sA}|$ defines a fibration of gonality $\gamma\geq 4$
		$$
		p_g(Y)\leq \frac{66p_g(X_b)-2q+8}{68}
		$$
		\item If $|K_{X_b}|$ is g.f.d. and $|M_{\sA}|$ defines a fibration of gonality $\gamma\geq 3$
		$$
		p_g(Y)\leq \frac{67p_g(X_b)-2q+8}{69}
		$$
		\item If $|K_{X_b}|$ is g.f.d. and $|M_{\sA}|$ defines a fibration of gonality $\gamma\geq 2$
		$$
		p_g(Y)\leq \frac{68p_g(X_b)-2q+7}{70}
		$$
        \item If $|K_{X_b}|$ defines a fibration of gonality $\gamma\geq 5$
        $$
        p_g(Y)\leq \frac{62p_g(X_b)+10}{67}
        $$
        \item If $|K_{X_b}|$ defines a fibration of gonality $\gamma\geq 4$
        $$
        p_g(Y)\leq \frac{16p_g(X_b)+2}{17}
        $$
        \item If $|K_{X_b}|$ defines a fibration of gonality $\gamma\geq 3$
        $$
        p_g(Y)\leq \frac{22p_g(X_b)+2}{23}
        $$
        \item If $|K_{X_b}|$ defines a fibration of gonality $\gamma\geq 2$
        $$
        p_g(Y)\leq \frac{34p_g(X_b)+2}{35}
        $$
    \end{itemize} 
If $\rank \sA=1$
$$
p_g(Y)\leq p_g(X_b)-1
$$ 
\end{prop}
\begin{proof}
	Proposition 4.3.2 in \cite{Ri} computes a list of all the upper bounds for the rank $u_f$. Our result follows immediately by noticing that $p_g(Y)\leq u_f$ since all the top forms on $Y$ are de Rham closed and hence their pullback restricted on the fiber is in the local system $\mD^{n-1}$ which we recall is the local system associated to the unitary flat vector bundle $\sU$, that is $\sU=\mD^{n-1}\otimes \sO_B$; see (\ref{seconda}).
	
	For the reader's convenience we briefly give an idea on how these bounds are obtained in \cite{Ri}. Consider the first case of this list, that is $|K_{X_b}|$ and $|M_{\sA}|$ are both g.f.n.d. The Harder-Narasimhan filtration of $f_*\omega_{X/B}$ is 
	$$
	0\subsetneq \sA\subsetneq f_*\omega_{X/B}
	$$ with $\mu_1=\deg f_*\omega_{X/B}/\rank \sA=\deg f_*\omega_{X/B}/(p_g(X_b)-u_f)$ and $\mu_2=0$ since $\sU$ is flat.
	
	The Xiao-Ohno-Konno formula then gives the inequality
	\begin{equation}
		\label{xok}
	K^3_f\geq \mu_1(M_{\sA}^2+M_{\sA}{K_{X_b}}+K_{X_b}^2).
	\end{equation} Thanks to \cite[Lemma 5.9]{Ba}, we get the necessary estimates  for the quantities appearing in (\ref{xok}) and we get 
	$$
	K^3_f\geq \frac{\deg f_*\omega_{X/B}}{p_g(X_b)-u_f}(3(p_g(X_b)-u_f)-7+3(p_g(X_b)-u_f)-6+3p_g(X_b)-7)
	$$ that easily gives the lower bound for the slope $\lambda^2_f$
	\begin{equation}
		\label{lambda}
			\lambda^2_f\geq 9+\frac{3u_f-20}{p_g(X_b)-u_f}
	\end{equation} see \cite[Theorem 4.2]{Ri}.
	
	The last step consists in using the inequality 
\begin{equation}
	K^3_f-2(g-1)K^2_{X_b}\leq 72\chi_f,
\end{equation} see \cite{O}.
This inequality for $g\leq 1$ gives 
$$
\lambda^1_f\leq 42
$$ 
hence remembering that under our hypothesis $\lambda^2_f\leq \lambda^1_f$ and putting together with (\ref{lambda}) we get 
\begin{equation}
	9+\frac{3u_f-20}{p_g(X_b)-u_f}\leq 72.
\end{equation} Isolating $u_f$ and using $p_g(Y)\leq u_f$ we get the first bound of the list.

The following bounds can be done in a similar way.
\end{proof}

\section{A result on the Albanese map}

This short final section is a result on the Albanese map which does  not directly follow by the previous work but it is in the same spirit. This is Theorem \ref{D}.
\begin{thm}
	\label{abel}
	Let $X$ be a smooth $n$-dimensional variety and $\alpha\colon X\to A:=Alb(X)$ its Albanese morphism. Assume that $\sL:=\Ima(\alpha^*\Omega^{n-1}_A\to \Omega^{n-1}_X)$ is a line bundle on $X$, then the global sections of $\sL$ define a rational map $h\colon X\dashrightarrow Y$ to a variety $Y$ of general type. Furthermore if the sections of $H^0(X,\sL)$ are $n-1$-strict, we can take $h$ to be a morphism and $Y$ is the Stein factorization of $X\to Z$ where $Z:=\alpha(X)$.
\end{thm}
\begin{proof}
	We begin by showing that $Z:=\alpha(X)$ is $n-1$-dimensional. For convenience consider $\alpha$ as the composition
	$$
	X \stackrel{\alpha'}{\longrightarrow} Z\stackrel{i}{\hookrightarrow} A.
	$$
	Since $\sL$ is not zero, it immediately follows that $\dim Z\geq n-1$. Similarly $\dim Z$ is not $n$ otherwise $\sL$ would be of rank $n$.
	
	Now we define a rational map $h\colon X\dashrightarrow Y$. Indeed the global sections of $\sL$ are $n-1$-forms with $\omega_i\wedge\omega_j=0$ since $\sL$ is a line bundle. We can then apply Theorem \ref{pcast} which defines our variety $Y$. In general we have $\dim Y\leq n-1$.
	
	Finally if the sections of $H^0(X,\sL)$ are $n-1$-strict, $\dim Y=n-1$, again by Theorem \ref{pcast}. To show that we have a rational map $Y\dashrightarrow Z$ we note that the kernel of the global sections of $\sL$ is a foliation as in \cite{RZ5}. More precisely, any global section of $\sL$ defines by contraction a map
	$$
	T_X\to \Omega^{n-2}_X
	$$ and since the sections are closed, the kernel is closed under Lie bracket and, up to saturation, gives a foliation.
	
	These leaves are contained in the fibers of both $\alpha'\colon X\to Z$ and $h\colon X\dashrightarrow Y$. Since $h$ has connected fibers we have that $Y$ is birational to the variety given by the Stein factorization of $\alpha'$. In particular it turns out that we can choose $h$ to be a morphism.
\end{proof}
The following corollary is in some sense a version of the Volumetric theorem \cite[Theorem 1.5.3]{PZ}
\begin{cor}
	Under the hypotheses of the Theorem \ref{abel}, if the restrictions of the Albanese map $\alpha$ to the fibers $X_b$ have degree 1, then these fibers are birational to $Y$.
\end{cor}
\begin{proof}
	By the previous Theorem it follows that $Y\to Z$ is a birational map, hence the fibers $X_b$ are in the same birational class as $Y$.
\end{proof}

\end{document}